\documentclass[10pt]{amsart} 

\usepackage{amsmath}
\usepackage{amssymb}

\usepackage{graphicx}
\usepackage{xcolor}
\usepackage{subcaption}
\usepackage{multirow}
\usepackage{enumitem}
\usepackage{mathrsfs}
\usepackage[rightcaption]{sidecap}
\usepackage{rotating}
\usepackage{tikz}

\newtheorem{theorem}{Theorem}[section]

\newtheorem{lemma}[theorem]{Lemma}

\theoremstyle{definition}

\theoremstyle{remark}
\newtheorem{remark}[theorem]{Remark}

\newcommand{\calN}{\mathcal{N}}
\newcommand{\calL}{\mathcal{L}}
\newcommand{\N}{\mathbb{N}}
\newcommand{\R}{\mathbb{R}}
\newcommand{\bfD}{\mathbf{D}}
\newcommand{\bfR}{\mathbf{R}}
\newcommand{\bfn}{\mathbf{n}}
\newcommand{\bfx}{\mathbf{x}}
\newcommand{\bfu}{\mathbf{u}}
\newcommand{\bfxi}{\boldsymbol{\xi}}

\newcommand{\resid}{\text{r}}
\newcommand{\ICBC}{\text{ICBC}}

\title[A scaled TW-PINN for traveling wave solutions]{A scaled TW-PINN: A physics-informed neural network for traveling wave solutions of reaction-diffusion equations with general coefficients}

\author{Seungwan Han}
\address{Department of Mathematics $\&$ POSTECH MINDS (Mathematical Institute for Data Science), Pohang University of Science and Technology, Pohang 37673, Korea}
\email{han97@postech.ac.kr}

\author{Kwanghyuk Park}
\address{Graduate School of Artificial Intelligence $\&$ POSTECH MINDS (Mathematical Institute for Data Science), Pohang University of Science and Technology, Pohang 37673, Korea}
\email{pkh0219@postech.ac.kr}

\author{Jiaxi Gu}
\address{Department of Mathematics $\&$ POSTECH MINDS (Mathematical Institute for Data Science), Pohang University of Science and Technology, Pohang 37673, Korea}
\email{jiaxigu@postech.ac.kr}

\author{Jae-Hun Jung}
\address{Department of Mathematics $\&$ POSTECH MINDS (Mathematical Institute for Data Science), Pohang University of Science and Technology, Pohang 37673, Korea}
\email{jung153@postech.ac.kr}

\makeatletter
\@namedef{subjclassname@2020}{\textup{}2020 Mathematics Subject Classification}
\makeatother
\subjclass[2020]{35K57, 68T07}
\keywords{Physics-informed neural network, Reaction-diffusion equation, Traveling wave solution, Scaling transformation, Wave speed, Universal approximation property}

\begin{document}

\maketitle

\begin{abstract}
We propose an efficient and generalizable physics-informed neural network (PINN) framework for computing traveling wave solutions of $n$-dimensional reaction-diffusion equations with various reaction and diffusion coefficients. 
By applying a scaling transformation with the traveling wave form, the original problem is reduced to a one-dimensional scaled reaction-diffusion equation with unit reaction and diffusion coefficients. 
This reduction leads to the proposed framework, termed scaled TW-PINN, in which a single PINN solver trained on the scaled equation is reused for different coefficient choices and spatial dimensions. 
We also prove a universal approximation property of the proposed PINN solver for traveling wave solutions. 
Numerical experiments in one and two dimensions, together with a comparison to the existing wave-PINN method, demonstrate the accuracy, flexibility, and superior performance of scaled TW-PINN. 
Finally, we explore an extension of the framework to the Fisher's equation with general initial conditions.
\end{abstract}

\section{Introduction}
Reaction-diffusion systems constitute a fundamental class of partial differential equations (PDEs), characterizing the interplay between the local reaction kinetics and the diffusive transport. 
They play a central role in physics, chemistry, biology, and ecology, providing a unifying framework for a wide range of phenomena such as excitable media \cite{Field,WinfreeSpiral}, catalytic surface reactions \cite{Jakubith,Ertl}, neuroscience \cite{Hodgkin,FitzHugh}, and population dynamics \cite{Fisher,Skellam}.
Their broad applicability stems from their capability to capture the core mechanisms underlying diverse spatiotemporal processes in natural systems.

The general reaction-diffusion system takes the form
\begin{equation*}
\partial_t \bfu = \nabla \cdot (\bfD \nabla \bfu) + \bfR(\bfu), 
\end{equation*}
where $\bfD$ is a diagonal matrix of diffusion coefficients and $\bfR(\bfu)$ denotes a nonlinear reaction term, typically involving reaction coefficients.
Depending on the reaction term and the diffusion coefficients, the reaction-diffusion system exhibits a variety of qualitatively distinct solution behaviors.
These include traveling waves \cite{Fisher,Saarloos,Xin}, spiral/scroll waves \cite{WinfreeSpiral,WinfreeScroll,Keener}, pattern formation \cite{Turing,Ouyang,MurrayII} and multistable phenomena \cite{Kauffman,Aronson,Fife}.
Within each class, key qualitative features of the solutions, including spatial profiles, characteristic length scales, and steepness, are governed by the diffusion coefficients in $\bfD$ and the parameters in $\bfR(\bfu)$, particularly reaction coefficients.

Physics-informed neural networks (PINNs) \cite{PINN} have recently become a widely used machine learning approach for approximating the solutions of PDEs.
However, PINNs often struggle to resolve solutions with sharp transitions \cite{PINN_fail2021, WangTeng}, a limitation characterized by slow convergence and inefficient learning.
In response to this limitation, some PINN variants have been proposed for discontinuous and sharp solutions \cite{MCCLENNY, LiuLiuXie, vsPINN}.
In \cite{wavePINN}, the wave-PINN method was proposed by introducing a residual weighting scheme to improve the PINN approximation of the sharp traveling wave in Fisher's equation with large reaction coefficients.
Despite this improvement, when the reaction coefficient becomes extremely large, wave-PINN fails to capture the wave front accurately, especially with respect to the predicted wave speed, as shown later in Section \ref{sec:comparison}.
Therefore, it remains challenging to accurately compute the sharp traveling wave front of the reaction-diffusion equation with PINNs when the reaction coefficient is large.

To address this issue of resolving the sharp wave front in the $n$-dimensional reaction-diffusion equation, we first apply a scaling transformation that normalizes the reaction and diffusion coefficients to unity. 
In the resulting scaled equation, the traveling wave profile becomes less steep, making the problem more suitable for PINN training.
By further exploiting the traveling wave form, the original problem is reduced to a one-dimensional scaled reaction-diffusion equation. 
With this reduction, even a PINN with simple architecture can be trained efficiently on the scaled equation.
The proposed PINN solver thus consists only of a wave layer, a single hidden layer, and an output layer.
The wave layer contains a trainable parameter corresponding to the predicted wave speed, enabling direct monitoring of wave speed learning during training and serves as an indicator of whether a training run achieves physical convergence.
Once the PINN solver is trained on the one-dimensional scaled equation, we combine it with the scaling and inverse transformations to construct a scaling PINN framework, referred to as scaled TW-PINN, for computing traveling wave solutions of $n$-dimensional reaction-diffusion equations with different reaction and diffusion coefficients.
Numerical experiments show that our scaled TW-PINN accurately captures the sharp wave front.

The outline of this paper is as follows. 
In Section \ref{sec:method}, we introduce the isotropic scalar reaction-diffusion equation, present the scaling transformation, and review the dimension-independent traveling wave reduction. 
We also summarize four representative reaction terms, each with a corresponding exact traveling wave solution and associated special wave speed.
Section \ref{sec:PINN} proposes the PINN solver, including its architecture, loss functions, and training configuration with convergence analysis, followed by the solution pipeline of the overall scaling PINN framework. 
It is shown that the proposed PINN solver has a universal approximation property for traveling wave solutions in Section \ref{sec:UAT}.
We provide numerical experiments in one and two dimensions, together with a comparison to the existing wave-PINN method, in Section \ref{sec:result}. 
Section \ref{sec:general} further explores an extension of the framework to Fisher's equation with general initial conditions.
The concluding section gives some remarks and outlines future developments.

\section{Scaling and traveling wave form of reaction-diffusion equations} \label{sec:method}
This section introduces the reaction-diffusion equation and its dimension-independent traveling wave reduction.
A key step from the original equation to its reduced form is the scaling transformation, which removes the explicit dependence on the reaction and diffusion coefficients and, more importantly, decreases the sharpness of the wave front. 
As a result, the scaled equation is better suited for PINN training.

\subsection{Reaction-diffusion equations}
\label{sec:RD}
Consider the isotropic scalar reaction-diffusion equation, in which the diffusion coefficient is identical in all spatial directions.
The governing equation takes the form
\begin{equation}
\label{eq:RD}
\partial_t u(\bfx,t) = D \nabla_{\bfx}^2 u(\bfx,t) + R(u(\bfx,t)),
\end{equation}
where $\bfx = (x_1, \cdots, x_n)$ with $n$ the spatial dimension, $D>0$ denotes the diffusion coefficient and $R(u)$ represents the reaction term.
The reaction term is assumed to have the general form \cite{Gu}
\begin{equation*}
R(u) = \rho \, u^p (1 - u^q)(u-a)^r,
\end{equation*}
where $p,q >0$, $ r \in \{0,1\}$, and $a\in(0,1)$.
The exponent $r=0$ corresponds to the monostable nonlinearity, while $r=1$ yields the classical bistable nonlinearity.
The parameter $\rho>0$ is the reaction coefficient that controls the overall strength of the nonlinear kinetics.
In \eqref{eq:RD}, increasing $\rho$ accelerates the local growth or decay of $u$ and hence shortens the characteristic reaction time scale.
Conversely, a decrease in $\rho$ reverses this trend.

\subsection{Scaling}
\label{sec:scaling}
We apply a scaling transformation to \eqref{eq:RD} by writing 
\begin{equation}
\label{eq:scaling} 
\tau = \rho \, t, \qquad \bfxi = \sqrt{\frac{\rho}{D}} \, \bfx, \qquad v(\bfxi,\tau) = u(\bfx,t).
\end{equation}
Then \eqref{eq:RD} becomes 
\begin{equation}
\label{eq:RD_scaling}
\partial_{\tau} v = \nabla_{\bfxi}^2 \, v + v^p (1-v^q)(v-a)^r,
\end{equation}
where both the reaction and diffusion coefficients become one.
It is easy to recover the original variables by the inverse transformation,
\begin{equation}
\label{eq:scaling_inverse}
t = \frac{1}{\rho} \tau, \qquad \bfx = \sqrt{\frac{D}{\rho}} \, \bfxi, \qquad u(\bfx,t) = v(\bfxi,\tau).
\end{equation}
Consequently, once the scaled equation \eqref{eq:RD_scaling} is solved, the solution of the original equation \eqref{eq:RD} is obtained by recovering $u$ from the scaled solution $v$.

\subsection{Dimension-independent traveling wave reduction}
\label{sec:traveling_wave_reduction}
A distinguished class of solutions to \eqref{eq:RD_scaling} is given by traveling waves of the form
\begin{equation}
\label{eq:traveling_wave}
v(\bfxi,\tau) = V(\zeta), \quad \zeta = \bfn \cdot \bfxi - c \tau, 
\end{equation}
where $\zeta$ is the analytical traveling wave coordinate, $\bfn$ denotes a unit vector specifying the direction of propagation and $c$ is the wave speed. 
Under the traveling wave form \eqref{eq:traveling_wave}, the scaled reaction-diffusion equation \eqref{eq:RD_scaling} reduces to the ordinary differential equation (ODE),
\begin{equation}
\label{eq:RD_scaling_ODE}
V'' + c V' + V^p (1-V^q)(V-a)^r = 0,
\end{equation}
where primes denote differentiation with respect to $\zeta$.
This reduction shows that, regardless of the spatial dimension, the traveling wave form leads to the same second-order ODE for the solution $V$.
As a result, the traveling wave solution $V$, as well as the speed $c$, is the same in any spatial dimension.
Thus the traveling wave form is intrinsically one-dimensional, even in higher dimensions.
This property justifies the use of a single PINN solver for computing traveling wave solutions of the $n$-dimensional reaction-diffusion equation.

\subsection{Reaction terms and corresponding exact traveling wave solutions}
\label{sec:reaction_term}
We study four representative reaction-diffusion equations that possess traveling wave solutions: Fisher's, Newell-Whitehead-Segel (NWS), Zeldovich, and bistable equations.
Each equation is characterized by a specific nonlinear reaction term $R(u)$ listed in Table \ref{tab:reaction}.
It is worth noting that the NWS equation with $q=2$ is the Allen-Cahn equation.
The spatially homogeneous equilibrium states are determined solely by the reaction term $R(u)$.
A typical profile of the traveling wave solution $V$ is monotone and connects two equilibrium states.
Specifically, $V(\zeta)$ satisfies
\begin{equation}
\label{eq:equil_state}
\lim_{\zeta \to -\infty} V(\zeta) = v_-, \qquad \lim_{\zeta \to \infty} V(\zeta) = v_+,
\end{equation}
so that the traveling wave represents a transition from the equilibrium state $v_-$ to the other equilibrium state $v_+$.
For the Fisher’s, NWS, and Zeldovich equations, the equilibrium states are $(v_-, v_+) = (1, 0)$, whereas for the bistable equation, they are $(v_-, v_+) = (a, 1)$.
\begin{table}[ht]
\centering
\caption{Reaction terms $R(u)$ for Fisher's, NWS, Zeldovich, and bistable equations.}
\label{tab:reaction}
\begin{tabular}{l c c}
\hline
          & $R(u)$              & Parameter range \\ 
\hline
Fisher    & $\rho\,u(1-u)$      & -- \\
NWS       & $\rho\,u(1-u^q)$    & $q > 0$ \\ 
Zeldovich & $\rho\,u^2(1-u)$    & -- \\
bistable  & $\rho\,u(1-u)(u-a)$ & $0<a<1$ \\
\hline
\end{tabular}
\end{table}

All four equations admit the exact traveling wave solution in closed form with a special wave speed $c$, as summarized in Table \ref{tab:RD_exact}.
These explicit solutions reveal that the steepness of the wave front scales proportionally to $\sqrt{\rho/D}$.
In particular, large values of $\rho/D$ lead to the thin transition layer with pronounced gradients, which poses significant challenges for both traditional numerical methods \cite{Gu} and neural network approaches \cite{wavePINN}.
\begin{table}[ht]
\centering
\caption{Special wave speed $c$ and its corresponding closed-form traveling wave solution $u(\bfx,t)$ for Fisher's, NWS, Zeldovich, and bistable equations.}
\label{tab:RD_exact}
\begin{tabular}{c c c}
\hline
          & $c$                                       & $u(\bfx,t)$ \\ 
\hline
Fisher    & $5 \sqrt{\frac{\rho D}{6}}$               & $\frac{1}{\left\{ 1+\exp{\left[\sqrt{\frac{\rho}{6D}} (\bfn \cdot \bfx-ct) \right]} \right\}^2 }$ \\
NWS       & $\frac{q+4}{\sqrt{2q+4}} \sqrt{\rho D} $  & $\left\{ \frac{1}{2} +\frac{1}{2} \tanh{\left[ -\frac{q}{2\sqrt{2q+4}}\sqrt{ \frac{\rho}{D}} (\bfn \cdot \bfx-ct) \right]}  \right\}^{\frac{2}{q}}$ \\
Zeldovich & $\sqrt{\frac{\rho D}{2}}$                 & $\frac{1}{1+\exp{\left[\sqrt{\frac{\rho}{2D}} (\bfn \cdot \bfx-ct) \right]}}$ \\
bistable  & $-(1+a)\sqrt{\frac{\rho D}{2}}$           & $\frac{1+a}{2} + \frac{1-a}{2} \tanh{\left[ \frac{1-a}{4}\sqrt{ \frac{2\rho}{D}} (\bfn \cdot \bfx-ct) \right]}$ \\
\hline
\end{tabular}
\end{table}

\section{Scaling PINN framework for traveling wave solutions} \label{sec:PINN}
Using the scaling transformation \eqref{eq:scaling}, and the traveling wave coordinate \eqref{eq:traveling_wave} for the traveling wave solution, we can simplify the task of solving a class of $n$-dimensional reaction-diffusion equations \eqref{eq:RD} with the same reaction term but different coefficients to the scaled one-dimensional reaction-diffusion equation,
\begin{equation}
\label{eq:RD_scaling_1D}
\partial_{\tau} v = \partial_{\xi \xi} v + v^p (1-v^q)(v-a)^r.
\end{equation}
To compute the traveling wave solution to this reaction-diffusion equation \eqref{eq:RD_scaling_1D}, we introduce a physics-informed neural network (PINN) solver.

\subsection{Architecture}
\label{sec:architecture}
The proposed solver adopts a PINN architecture, illustrated in Fig. \ref{fig:NN_model}.
\begin{figure}[ht!]
\centering
\includegraphics[width=0.5\columnwidth]{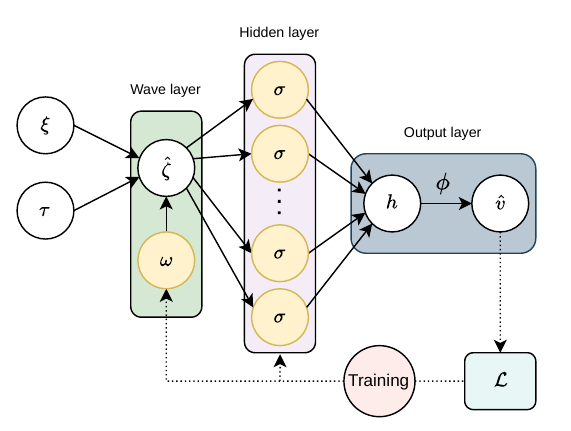}
\caption{Schematic of the PINN architecture.}
\label{fig:NN_model}
\end{figure}
The architecture consists of three layers: a wave layer \cite{Cho, wavePINN} that imposes the traveling wave form, a single hidden layer responsible for function approximation in accordance with the universal approximation theorem, and an output layer that enforces the equilibrium states.

The network takes ($\xi, \tau$) as input.
Motivated by \eqref{eq:traveling_wave}, the wave layer introduces the predicted traveling wave coordinate in one-dimensional form, 
\begin{equation}
\label{eq:wave_coordinate_predict}
\hat{\zeta} = \xi - \omega \, \tau,
\end{equation}
where $\omega$ is a trainable parameter representing the predicted wave speed. 
This parameter allows the predicted wave speed to be monitored directly during training and is used in Section \ref{sec:training} as a diagnostic for physical convergence.
It should be noted that the wave layer used for training here is extended to multiple spatial dimensions in \eqref{eq:wave_layer} employed for the scaling PINN framework in Section \ref{sec:pipeline}.
The hidden layer is implemented as a fully connected layer with $N$ neurons, producing the output
\begin{equation*}
h(\hat{\zeta}) = \sum_{i=1}^{N} c_i \, \sigma(a_i \hat{\zeta} + b_i),
\end{equation*}
where $\sigma$ denotes a sigmoid activation function, defined as any bounded, differentiable real-valued function on $\R$ with strictly positive derivative.
In this paper, our network employs the logistic sigmoid function for $\sigma$.
To ensure that the solution remains within the equilibrium states $(v_-,v_+)$, an output constraint is imposed.
Since the traveling wave solution lies in the interval $(v_-,v_+)$, we introduce the constraint function $\phi$, defined by
\begin{equation}
\label{eq:constraint}
\phi (s)= v_{-} + (v_{+} - v_{-})\frac{1}{1+\exp{(-s)}} , \quad s \in \R.
\end{equation}
This construction guarantees that $\phi(s) \in (v_{-}, v_{+})$ for all $s \in \R$, which eliminates spurious oscillations around the wave front.
Moreover, as $\phi$ is monotone, the network inherits this monotonicity.
The final network output is then given by
\begin{equation}
\label{eq:model}
\hat{v}(\xi, \tau) 
= \phi\!\left( \sum_{i=1}^{N} c_i \, \sigma(a_i \hat{\zeta} + b_i) \right).
\end{equation}

In our proposed network, only the predicted wave speed $\omega$ and the weights of the hidden layer, indicated in yellow in Fig. \ref{fig:NN_model}, are trainable, resulting in a compact and structurally constrained architecture.
As shown in Section \ref{sec:UAT}, this network has the universal approximation property for traveling wave solutions while preserving the prescribed equilibrium states.

\subsection{Loss functions}
\label{sec:loss}
From \eqref{eq:RD_scaling_1D}, the reaction-diffusion equation can be rewritten as
\begin{equation*}
\partial_{\tau} v + \calN[v] = 0, \quad \calN[v] = - \partial_{\xi \xi} - v^p (1-v^q)(v-a)^r,
\end{equation*}
subject to the initial condition $v(\xi,0) = v_0(\xi)$ and some boundary conditions.
In the PINN solver, the governing equation and data constraints are incorporated directly into the loss function.
This approach restricts the approximation space so that the learned solution satisfies the equation together with the prescribed conditions.

Following \cite{PINN}, the loss is given by
\begin{equation*}
\calL = \calL_{\ICBC} + \calL_{\resid}.
\end{equation*}
where
\begin{align*}
\calL_{\ICBC} &= \frac{1}{N_{\ICBC}} \sum_{i=1}^{N_{\ICBC}} \left| \hat{v}(\xi_v^i,\tau_v^i) - v^i \right|^2, \\
\calL_{\resid} &= \frac{1}{N_{\resid}} \sum_{i=1}^{N_{\resid}} \left| \hat{v}_{\tau}(\xi_{\resid}^i,\tau_{\resid}^i) + \calN \! \left[\hat{v}(\xi_{\resid}^i,\tau_{\resid}^i)\right] \right|^2.
\end{align*}
Here, $\{\xi_v^i, \tau_v^i, v^i\}_{i=1}^{N_\ICBC}$ denote the initial and boundary training data on $v(\xi,\tau)$ and $\{\xi_{\resid}^i, \tau_{\resid}^i \}_{i=1}^{N_\resid}$ specify the collocations points for equation residual.

\subsection{Training configuration and convergence analysis}
\label{sec:training}
The training $(\xi, \tau)$-domain for the scaled reaction-diffusion equation \eqref{eq:RD_scaling_1D} is $[-5000,5000] \times [0,2000]$.
We use $N_{\ICBC} = 1024$ collocation points for the initial and boundary loss $\calL_{\ICBC}$, and $N_{\resid} = 1024$ collocation points for the residual loss $\calL_{\resid}$.
All collocation points are generated via Latin hypercube sampling.
The network parameters are initialized with different random seeds.
Training is performed for $100,000$ epochs using the Adam optimizer with an initial learning rate of $0.01$, together with cosine annealing to gradually decrease the learning rate throughout the training process.

Under the configuration specified above, most training runs, each utilizing a different random seed initialization, achieve physical convergence, while a few exhibit spurious convergence.
Figure \ref{fig:converge} shows representative physical and spurious convergent runs for Fisher’s equation by showing the evolution of the training loss $\calL$ and the predicted wave speed $\omega$.
Similar phenomena occur for the NWS, Zeldovich, and bistable equations; Fisher’s equation is presented as a representative example.
\begin{figure}[ht!]
\centering
\includegraphics[width=0.9\columnwidth]{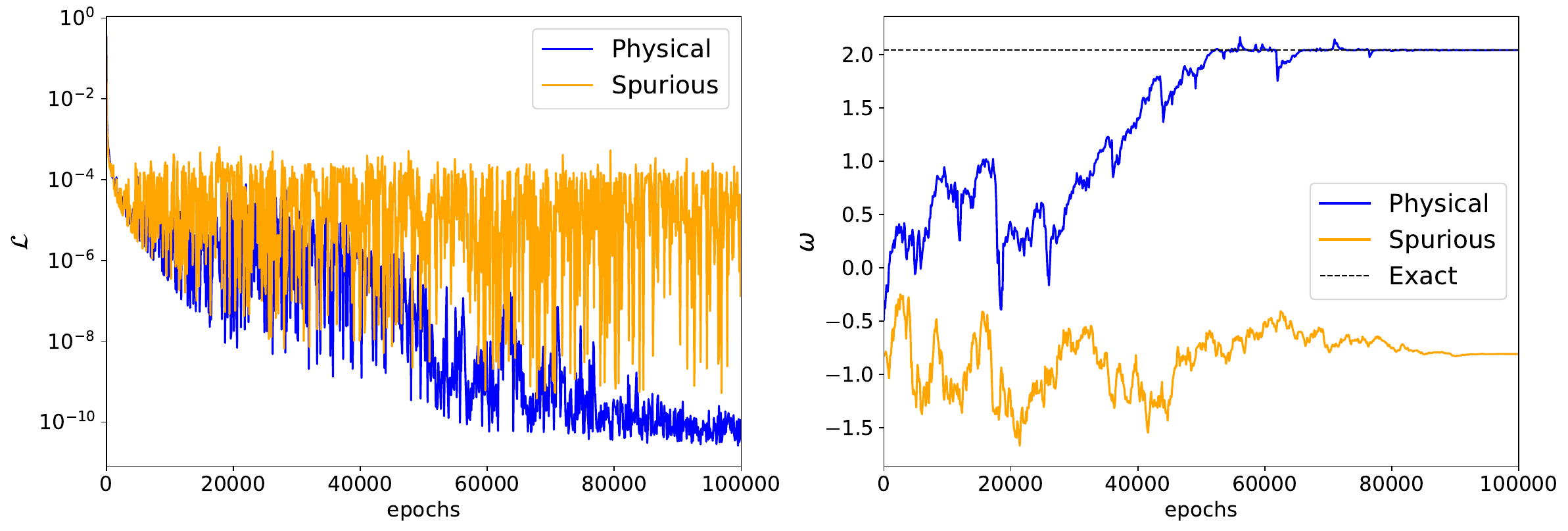}
\caption{Training behavior on the physical (blue) and spurious (orange) convergence for Fisher's equation: log-scale training loss $\calL$ (left) and predicted wave speed $\omega$ (right) compared to the exact wave speed (dashed).}
\label{fig:converge}
\end{figure}
In the physically convergent case, the training loss decays steadily toward $10^{-10}$ with only minor fluctuations.
In the mean while, the predicted wave speed $\omega$ exhibits oscillations of decreasing amplitude and gradually approaches the exact wave speed $c= \tfrac{5}{\sqrt{6}} \approx 2.04$.
In contrast, for spuriously convergent runs, the loss oscillates persistently within the range of $[10^{-9}, 10^{-4}]$ without stabilization, and the predicted wave speed remains separated from the exact wave speed.
This qualitative difference suggests that accurate identification of the wave speed plays an important role in physical convergence.
For each equation, training is therefore repeated with different random seeds until ten physical convergent solvers are obtained.
Those convergent solvers are used for the numerical experiments in Section \ref{sec:result}.
Table \ref{tab:wavespeed_diff} reports the discrepancy between the exact wave speed $c$ in \eqref{eq:traveling_wave} and the predicted wave speed $\omega$ in \eqref{eq:wave_coordinate_predict} for Fisher's, NWS $(q=2)$, Zeldovich, and bistable $(a=0.2)$ equations.
For each case, the mean absolute error $|c-\omega|$ is computed from those ten physically convergent runs, and the corresponding standard deviation is shown in parentheses. 
While Fisher's, Zeldovich, and bistable equations yield consistently small errors, the NWS equation exhibits comparatively large errors.
\begin{table}[ht!]
\centering
\caption{Absolute error in predicted wave speed on original and restricted domains.}
\label{tab:wavespeed_diff}
\resizebox{\textwidth}{!}{
\begin{tabular}{l c c c c c c}
\hline
                              &            & Fisher            & NWS $(q=2)$       & Zeldovich         & bistable $(a=0.2)$ \\
\hline
\multirow{2}{*}{$|c-\omega|$} & original   & 1.31e-4 (8.82e-5) & 2.94e-4 (2.39e-4) & 4.14e-5 (2.89e-5) & 4.23e-5 (4.95e-5) \\
                              & restricted & 1.91e-6 (1.54e-6) & 1.98e-6 (1.16e-6) & 1.73e-6 (2.24e-6) & 1.09e-6 (6.92e-7) \\
\hline
\end{tabular}}
\end{table}
We also observe that the physical convergence behavior varies across equations.
In particular, Fisher's, Zeldovich, and bistable equations predict the wave speed accurately under most random initializations, whereas the NWS equation frequently fails to predict correctly.

We further take a look at the case of spurious convergence.
The traveling wave solution connects the equilibrium states and propagates through the domain, so accurate learning of the wave speed requires sufficient resolution near the wave front.
However, the original training $(\xi, \tau)$-domain is too large ($[-5000,5000] \times [0,2000]$) relative to the number of collocation points ($N_{\ICBC}=1024$ and $N_{\resid}=1024$), resulting in a sparse sampling of the wave front region.
To improve resolution, two approaches are considered: increasing the number of collocation points or restricting the training domain while assuming that outer regions remain close to the equilibrium states.
Empirical observations indicate that domain restriction is significantly more effective.
We restrict the training $(\xi, \tau)$-domain to $[-500,500] \times [0,20]$.
This adjustment increases the effective sampling density near the wave front and concentrates training on the dynamically relevant region.
Figure \ref{fig:div} compares training on the original and restricted domains using the random seed, which fails to converge physically on the original domain.
\begin{figure}[ht!]
\centering
\includegraphics[width=0.9\columnwidth]{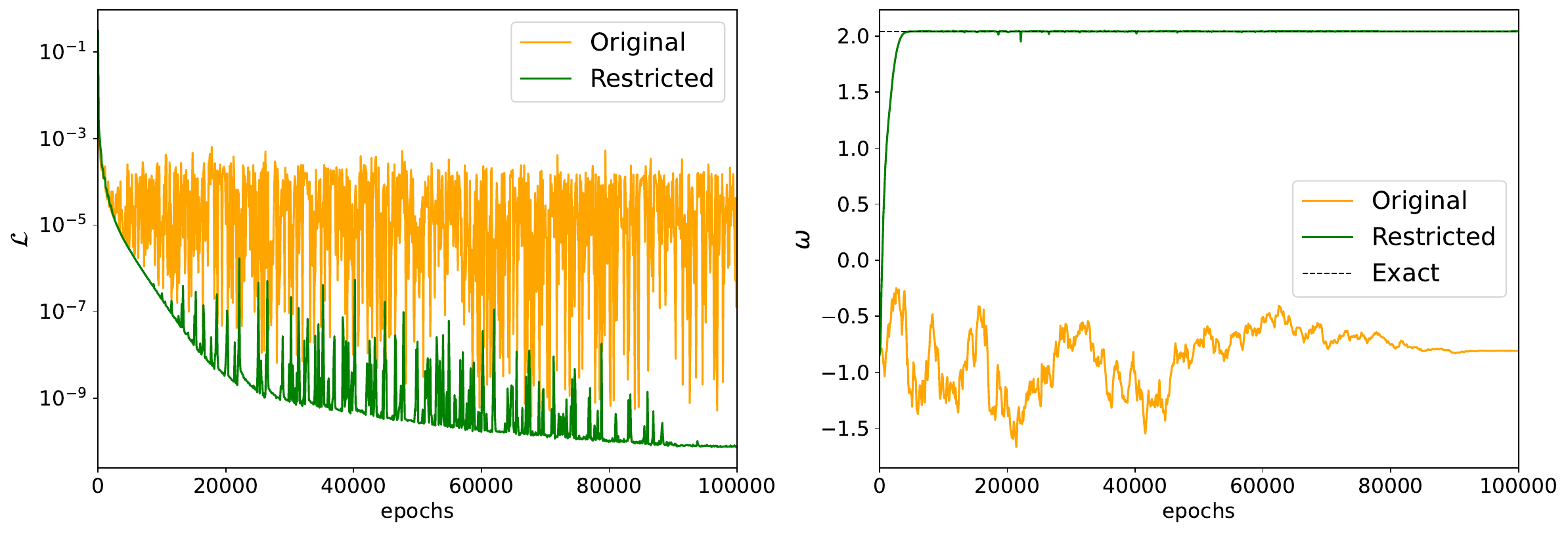}
\caption{Training behavior on the original (orange) and restricted (green) domains for Fisher's equation: log-scale training loss $\calL$ (left) and predicted wave speed $\omega$ (right) compared to the exact wave speed (dashed).}
\label{fig:div}
\end{figure}
On the original domain, the PINN solver fails to learn the correct wave speed, and the loss oscillates persistently within the range of $[10^{-9}, 10^{-4}]$ without stabilization, corresponding to the case of spurious convergence.
On the restricted domain, the predicted wave speed converges rapidly to the exact wave speed without oscillations, and the loss decreases smoothly, which corresponds to the physical convergence.
This behavior is observed for all tested random seeds.
These results show that increasing the effective sampling density near the wave front is crucial for accurate wave speed identification and stable training.
We also notice that physical convergence is faster than in the setting of the original domain.
For consistency, training is repeated using the same random seeds as in the original domain, which yielded convergent physical solutions.
Table \ref{tab:wavespeed_diff} also reports the discrepancy between the exact wave speed $c$ and the corresponding predicted wave speed $\omega$ on restricted domains.
A substantial improvement is observed on the restricted domain.
The improved accuracy of the approximation by the scaling PINN framework on the restricted domain is confirmed by the numerical results in Section \ref{sec:result}.

In addition, separate solvers are trained for the NWS equation on the restricted domain with $q=3$ and $q=4$, respectively.
As seen in Table \ref{tab:wavespeed_diff_nws}, the errors in the predicted wave speed remain on the order of $10^{-6}$.
These results demonstrate that PINN accurately learns the wave speed even for $q \geq 2$.
This robust performance stands in contrast to WENO schemes, which shows the difficulty in accurately capturing the wave speed \cite{Gu}.
\begin{table}[ht!]
\centering
\caption{Absolute error in the predicted wave speed for NWS equation $(q=3,4)$ on the restricted domain.}
\label{tab:wavespeed_diff_nws}
\begin{tabular}{c c c}
\hline
$q$ &  3 & 4\\
\hline
$|c-\omega|$ & 1.50e-6 & 3.63e-6  \\
\hline
\end{tabular}
\end{table}

\subsection{Solution pipeline}
\label{sec:pipeline}
After training the PINN solver in Section \ref{sec:training}, we incorporate it into the proposed scaling PINN framework to compute numerical solutions of the reaction-diffusion equation \eqref{eq:RD}.
Figure \ref{fig:pipeline} illustrates the solution pipeline of the framework.
The procedure consists of three stages: a scaling transformation, approximation of the scaled equation using the trained PINN solver, and an inverse transformation.
\begin{figure}[ht!]
\centering
\includegraphics[width=0.5\columnwidth]{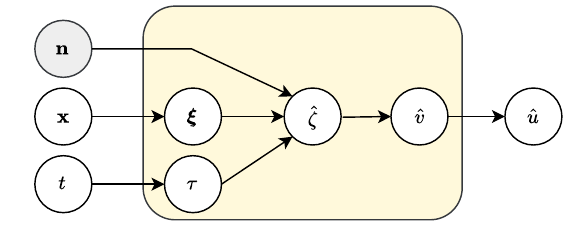}
\caption{Solution pipeline of the proposed scaling PINN framework: a scaling transformation, approximation of the scaled equation using a trained PINN solver, and an inverse transformation.}
\label{fig:pipeline}
\end{figure}

The framework takes $(\bfn,\bfx,t)$ as input.
The spatial variable $\bfx$ and temporal variable $t$ are first processed by a pre-processing layer that performs the scaling transformation \eqref{eq:scaling}, giving the scaled variables $\bfxi$ and $\tau$.
Correspondingly, the original reaction-diffusion equation is transformed into the scaled form in which the reaction and diffusion coefficients are rescaled to unity.
The resulting variables $(\bfn,\bfxi,\tau)$ are then passed through the wave layer,
\begin{equation}
\label{eq:wave_layer}
\hat{\zeta} = \bfn \cdot \bfxi -\omega\,\tau,
\end{equation}
with the learned wave speed $\omega$.
This layer operates in a dimensionally independent manner, motivated by the observation in Section \ref{sec:traveling_wave_reduction} that the traveling wave solution is independent of the spatial dimension.
The predicted traveling wave coordinate $\hat{\zeta}$ is then passed through the single hidden layer and the output constraint to compute the scaled solution $\hat{v}$ by the trained PINN solver.
This step corresponds to solving the scaled equation using the proposed PINN solver, represented by the central block in Fig. \ref{fig:pipeline}.
Finally, the approximate solution $\hat{v}$ is transformed back to $\hat{u}$ through the inverse transformation \eqref{eq:scaling_inverse}.
This step recovers the solution of the original reaction-diffusion equation with the given reaction and diffusion coefficients.
Since the scaling is linear, the inverse transformation is also linear and can therefore be computed efficiently.
This pipeline enables a single trained solver to consistently compute solutions for a wide range of coefficients and arbitrary spatial dimensions within a unified framework.

\section{Universal approximation property for traveling wave solutions} \label{sec:UAT}
In this section, we show that the architecture proposed in Section \ref{sec:architecture} possesses the universal approximation property for traveling wave solutions.

\begin{theorem}[Universal approximation theorem \cite{Cybenko}]
\label{thm:UAT}
Let $\sigma$ be any continuous discriminatory function.
Then finite sums of the form 
\begin{equation*}
G(\zeta) = \sum_{i=1}^{N} c_i \, \sigma(a_i^T \zeta + b_i)
\end{equation*}
are dense in $C(I_n)$.
In other words, given any $F \in C(I_n)$ and $\varepsilon>0$, there is a sum, $G(\zeta)$, of above form, for which
\begin{equation*}
\left| F(\zeta) - G(\zeta) \right| < \varepsilon
\quad \text{for all } \zeta \in I_n.
\end{equation*}
\end{theorem}

\begin{remark}
In the above theorem, $I_n=[0,1]^n \subset \R^n$ denotes the $n$-dimensional unit cube, and $a_i^T\zeta$ is the inner product of $a_i\in\R^n$ and $\zeta\in I_n$.
The universal approximation result stated on $C(I_n)$ extends to $C(K)$ for any compact set $K \subset \R^n$ by composing $F$ with an affine bijection between $K$ and a subset of $I_n$.
A function $\sigma$ is said to be discriminatory if, for a measure $\mu \in M(I_n)$, the space of finite signed regular Borel measures on $I_n$,
\begin{equation*}
    \int_{I_n} \sigma(a^T \zeta + b) \, \text{d}\mu(\zeta) = 0,
\end{equation*}
for all $a\in\R^n$ and $b \in \R$ implies that $\mu=0$.
The activation function employed in Section \ref{sec:architecture} satisfies the continuity and discriminatory assumptions of Theorem \ref{thm:UAT}.
Therefore, the corresponding shallow networks are dense in $C(K)$ for each compact set $K \subset \R^n$.
\end{remark}

The following lemma shows that this approximation property remains valid for one-dimensional functions under composition with a constraint function.
\begin{lemma}[Universal approximation under constraints]
\label{thm:UAT_C}
Let $K \subset \R$ be a compact set, and $V: K \to (\alpha, \beta)$ a continuous function.
Suppose $\sigma$ is a sigmoid function, and $\phi: \R \to (\alpha, \beta)$ is continuous, surjective, and strictly monotone.
Then for any $\varepsilon > 0$, there exist an integer $N \in \N$ and constants $a_i, b_i, c_i \in \R$ for $i = 1, \dots, N$ such that for all $\zeta \in K$,
\begin{equation*}
\left| V(\zeta) - \phi \left( \sum_{i=1}^{N} c_i \, \sigma(a_i \zeta + b_i) \right) \right| < \varepsilon.
\end{equation*}
\end{lemma}

\begin{proof}
Since $\phi$ is continuous and bijective, its inverse $\phi^{-1}$ exists and is continuous.
Define $F(\zeta) = \phi^{-1}(V(\zeta))$.
Since $V$ is continuous on $K$, the composition $F = \phi^{-1} \circ V$ is continuous on $K$.
Since $K$ is compact, the image $F(K)$ is compact and thus bounded.
Let $I_F$ be a compact interval such that $F(K)$ is contained in the interior of $I_F$.
The continuity of $\phi$ on the compact set $I_F$ implies the uniform continuity of $\phi|_{I_F}$.
Fix $\varepsilon >0$.
By the uniform continuity of $\phi|_{I_F}$, there exists $\delta >0$ independent of $\zeta \in K$ such that $(F(\zeta)-\delta,F(\zeta)+\delta) \subset I_F$ and if $|F(\zeta) - y|<\delta$, then
\begin{equation*}
|\phi(F(\zeta)) - \phi(y)| < \varepsilon,
\end{equation*}
for all $\zeta \in K$ and $ y \in I_F$.
Applying Theorem \ref{thm:UAT} with $n=1$ to the continuous function $F$ on the compact set $K$, there exist $N \in \N$ and constants $a_i, b_i, c_i \in \R$ such that
\begin{equation*}
\left| F(\zeta) - \sum_{i=1}^{N} c_i \, \sigma(a_i \zeta + b_i) \right| < \delta,
\end{equation*}
for all $\zeta \in K$.
Define
\begin{equation*}
\hat{V}(\zeta) = \phi \left( \sum_{i=1}^{N} c_i \, \sigma(a_i \zeta + b_i) \right).
\end{equation*}
By the choice of $\delta$,
\begin{equation*}
\left| V(\zeta) - \hat{V}(\zeta) \right| < \varepsilon \quad \text{for all } \zeta \in K,
\end{equation*}
which completes the proof.
\end{proof}

Using this lemma, we can now establish that the traveling wave solution can be represented by the proposed solver.
\begin{theorem}[Universal approximation for traveling wave solutions]
\label{thm:exist}
Let $\Omega \subset \R$ be a compact set and let $\hat{v} : \Omega \times [0,T] \to \R$ denote the solver defined in \eqref{eq:model}.
Then for any $\varepsilon > 0$, there exist an integer $N$, constants $a_i, b_i, c_i \in \R$ for $i = 1,\dots,N$ and a predicted wave speed $\omega \in \R$, such that for all $\xi \in \Omega, \ \tau \in [0,T]$,
\begin{equation*}
\left| v(\xi,\tau) - \hat{v}(\xi,\tau) \right| < \varepsilon,
\end{equation*}
where $v: \Omega \times [0,T] \to (v_-, v_+)$ is the traveling wave solution of the scaled reaction-diffusion equation
\begin{equation*}
v_{\tau} = v_{\xi \xi} + v^p (1 - v^q) (v-a)^r, \quad p,q>0,\; r\in\{ 0,1 \},\; a \in (0,1),
\end{equation*}
with wave speed $c$.
\end{theorem}

\begin{proof}
Define
\begin{equation*}
K = \left\{ \xi - c \tau \mid \xi \in \Omega, \ \tau \in [0,T] \right\}.
\end{equation*}
Then $K \subset \R$ is a compact. 
Let $V$ denote the traveling wave profile defined by
\begin{equation*}
V(\zeta) = v(\xi,\tau), \qquad \zeta=\xi-c\tau.
\end{equation*}
Since $v$ is a traveling wave solution, $V$ is well defined on $K$ and satisfies $V(\zeta) \in (v_-,v_+)$ for all $\zeta \in K$.
The function $\phi$ in \eqref{eq:constraint} is continuous, surjective and strictly monotone.
Thus, by Lemma \ref{thm:UAT_C} with $(\alpha, \beta) = (v_-,v_+)$, there exists a function
\begin{equation*}
\hat{V}(\zeta) = \phi \! \left( \sum_{i=1}^{N} c_i \, \sigma \! \left( a_i \, \zeta + b_i \right)\right)
\end{equation*}
such that for all $\zeta \in K$,
\begin{equation*}
\left| V(\zeta) - \hat{V}(\zeta) \right| < \varepsilon.
\end{equation*}
Define $\hat{v}(\xi, \tau) = \hat{V}(\xi - c \tau)$ and choose $\omega=c$.
Then for all $\xi \in \Omega$ and $\tau \in [0,T]$,
\begin{equation*}
\left|v(\xi, \tau) - \hat{v}(\xi,\tau) \right| < \varepsilon.
\end{equation*}
\end{proof}
Theorem \ref{thm:exist} shows that the proposed architecture can approximate traveling wave solutions of the scaled reaction-diffusion equation with arbitrary accuracy, provided that the predicted wave speed $\omega$ coincides with the exact wave speed $c$.

\section{Numerical results for scaled TW-PINN} \label{sec:result}
We present one- and two-dimensional numerical results to demonstrate the potential of the proposed scaling PINN framework for traveling wave solutions, referred to as scaled TW-PINN.
The diffusion coefficient is fixed as $D=1$ in this section.
To assess the accuracy of the scaled TW-PINN, we make use of $L_2$ and $L_\infty$ error norms:
\begin{align*}
 & L_2 = \sqrt{\frac{1}{N} \sum_{i=1}^N \left( u^{(i)} - \hat{u}^{(i)} \right)^2}, \\
 & L_{\infty} = \max_{1 \leqslant i \leqslant N} \left| u^{(i)} - \hat{u}^{(i)} \right|,
\end{align*}
where $u^{(i)}$ denotes the exact solution and $\hat{u}^{(i)}$ denotes the PINN approximation at the $i$-th collocation point of $N$ samples.

\subsection{One-dimensional numerical results}
\label{sec:result1D}
Since the diffusion coefficient $D$ is fixed, the spatial and temporal domains for numerical experiments depend only on the reaction coefficient $\rho$.
Table \ref{tab:domain} summarizes the domains for each reaction-diffusion equation and each value of $\rho$, with the final time chosen so that the traveling wave front remains within the spatial domain.
\begin{table}[ht!]
\centering
\caption{Spatial and temporal domains for each reaction-diffusion equation and each value of $\rho$.}
\label{tab:domain}
\begin{tabular}{l c c c c c}
\hline
          & variable & \multicolumn{4}{c}{$\rho$}   \\
                       \cline{3-6}
          &          & $1$         & $10^2$     & $10^4$      & $10^6$  \\
\hline
Fisher    &  $x,y$   & $[ -5, 25]$ & $[-1, 5]$  & $[-1, 5]$   & $[-1, 5]$ \\
          &  $t$     & $[  0, 10]$ & $[0, 0.2]$ & $[0, 0.02]$ & $[0, 0.002]$ \\
\hline
NWS       &  $x,y$   & $[ -5, 25]$ & $[-1, 5]$  & $[-1, 5]$   & $[-1, 5]$ \\
          &  $t$     & $[  0, 10]$ & $[0, 0.2]$ & $[0, 0.02]$ & $[0, 0.002]$ \\
\hline
Zeldovich &  $x,y$   & $[ -5, 25]$ & $[-1, 5]$  & $[-1, 5]$   & $[-1, 5]$ \\
          &  $t$     & $[  0, 30]$ & $[0, 0.6]$ & $[0, 0.06]$ & $[0, 0.006]$ \\
\hline
bistable  &  $x,y$   & $[-25, 5]$ & $[-5, 1]$  & $[-5, 1]$   & $[-5, 1]$ \\
          &  $t$     & $[ 0, 25]$ & $[0, 0.5]$ & $[0, 0.05]$ & $[0, 0.005]$ \\
\hline
\end{tabular}
\end{table}

For each value of $\rho$, $500$ collocation points are uniformly sampled in both space and temporal time, giving a total of $250,000$ collocation points.
The accuracy test uses ten physically convergent PINN solvers, with different random seed initializations, trained on the original and restricted domains, as described in Section \ref{sec:training}.
The mean and standard deviation of the $L_2$ and $L_{\infty}$ errors for Fisher's, NWS, Zeldovich and bistable equations at different values of $\rho$ are listed in Tables \ref{tab:L2_1D} and \ref{tab:Linf_1D}, respectively.
In general, the error increases as $\rho$ becomes larger.
However, on the restricted domain, NWs and bistable equations achieves the smallest $L_2$ error at $\rho=10^4$.
The scaled TW-PINNs on the restricted domain consistently give more accurate solutions, which can be attributed to their improved estimates of the predicted wave speed $\omega$, as shown in Section \ref{sec:training}.
Notably, for each equation, the $L_2$ errors calculated from scaled TW-PINNs on the restricted domain at $\rho=10^6$ are smaller than those from scaled TW-PINNs on the original domain at $\rho=1$.
These observations indicate that improved estimation of $\omega$ leads to more reliable numerical solutions.
Moreover, at $\rho = 10^6$, the errors by scaled TW-PINNs on the restricted domain are of comparable magnitude across all equations.
In contrast, for scaled TW-PINNs on the original domain, the NWS equation exhibits significantly larger errors than the other equations.
\begin{table}[ht!]
\centering
\caption{$L_2$ error for one-dimensional Fisher's, NWS $(q=2)$, Zeldovich and bistable $(a=0.2)$ equations.}
\label{tab:L2_1D}
\resizebox{\textwidth}{!}{
\begin{tabular}{l l c c c c}
\hline
          &            & \multicolumn{4}{c}{$\rho$}   \\
                         \cline{3-6}
          &            & $1$ & $10^2$ & $10^4$ & $10^6$ \\
\hline
Fisher    & original   & 5.21e-5 (2.41e-5) & 6.60e-5 (3.46e-5) & 1.76e-4 (1.18e-4) & 5.55e-4 (3.75e-4) \\
          & restricted & 5.33e-6 (1.52e-6) & 7.52e-6 (1.77e-6) & 6.49e-6 (1.54e-6) & 1.04e-5 (5.36e-6) \\
\hline
NWS       & original   & 1.13e-4 (8.40e-5) & 1.54e-4 (1.20e-4) & 4.75e-4 (3.86e-4) & 1.50e-3 (1.22e-3) \\
          & restricted & 6.02e-6 (2.37e-6) & 6.14e-6 (2.21e-6) & 5.67e-6 (2.05e-6) & 1.12e-5 (5.89e-6) \\
\hline
Zeldovich & original   & 6.55e-5 (2.44e-5) & 7.69e-5 (3.71e-5) & 2.02e-4 (1.39e-4) & 6.37e-4 (4.45e-4) \\
          & restricted & 1.00e-5 (4.50e-6) & 1.01e-5 (4.53e-6) & 1.01e-5 (8.15e-6) & 2.18e-5 (2.70e-5) \\
\hline
bistable  & original   & 5.04e-5 (3.09e-5) & 5.54e-5 (4.17e-5) & 1.24e-4 (1.42e-4) & 3.88e-4 (4.53e-4) \\
          & restricted & 6.77e-6 (2.62e-6) & 7.41e-6 (2.37e-6) & 6.62e-6 (1.97e-6) & 1.15e-5 (5.54e-6) \\
\hline
\end{tabular}}
\end{table}
\begin{table}[ht!]
\centering
\caption{$L_{\infty}$ error for one-dimensional Fisher's, NWS $(q=2)$, Zeldovich and bistable $(a=0.2)$ equations.}
\label{tab:Linf_1D}
\resizebox{\textwidth}{!}{
\begin{tabular}{l l c c c c}
\hline
          &            & \multicolumn{4}{c}{$\rho$}   \\
                         \cline{3-6}
          &            & $1$ & $10^2$ & $10^4$ & $10^6$ \\
\hline
Fisher    & original   & 1.80e-4 (9.90e-5) & 3.27e-4 (2.07d-4) & 3.14e-3 (2.12e-3) & 3.06e-2 (2.07e-2) \\
          & restricted & 1.82e-5 (5.70e-6) & 1.91e-5 (5.13e-6) & 5.02e-5 (3.60e-5) & 4.47e-4 (3.60e-4) \\
\hline
NWS       & original   & 5.72e-4 (3.74e-4) & 1.14e-3 (7.62e-4) & 1.15e-2 (7.72e-3) & 1.12e-1 (7.35e-2) \\
          & restricted & 1.39e-5 (5.52e-6) & 1.42e-5 (5.37e-6) & 8.10e-5 (4.83e-5) & 8.01e-4 (4.83e-4) \\
\hline
Zeldovich & original   & 2.42e-4 (1.27e-4) & 4.39e-4 (2.82e-4) & 4.30e-3 (2.95e-3) & 4.29e-2 (2.94e-2) \\
          & restricted & 2.65e-5 (1.32e-5) & 2.79e-5 (1.67e-5) & 1.33e-4 (1.83e-4) & 1.31e-3 (1.83e-3) \\
\hline
bistable  & original   & 1.58e-4 (1.28e-4) & 2.64e-4 (2.67e-4) & 2.39e-3 (2.79e-3) & 2.37e-2 (2.77e-2) \\
          & restricted & 1.73e-5 (7.68e-6) & 1.75e-5 (7.52e-6) & 6.04e-5 (3.91e-5) & 5.71e-4 (4.18e-4) \\
\hline
\end{tabular}}
\end{table}

As seen in Section \ref{sec:training}, the scaled TW-PINN can accurately predict the wave speed for the NWS equation even when $q \geq 2$.
Table \ref{tab:nws34} calculates the $L_2$ and $L_{\infty}$ errors for the NWS equations with $q = 3, 4$ under different values of $\rho$.
These results are computed using one PINN solver for each value of $q$ trained on the restricted domain.
The resulting errors are significantly smaller than those by the central WENO scheme \cite{Gu}, where the numerical solution exhibits a noticeable lag behind the exact solution for $q \geq 2$.
This comparison further indicates that the scaled TW-PINNs on the restricted domain capture the wave front more accurately.
\begin{table}[ht!]
\centering
\caption{$L_2$ and $L_{\infty}$ errors for NWS equation $(q=3,4)$ using scaled TW-PINNs on the restricted domain.}
\label{tab:nws34}
\begin{tabular}{l l c c c c}
\hline
$q$ & norm         & \multicolumn{4}{c}{$\rho$}     \\
                     \cline{3-6}
    &              & $1$     & $10^2$  & $10^4$  & $10^6$  \\
\hline
3   & $L_2$        & 6.39e-6 & 5.78e-6 & 4.39e-6 & 8.81e-6  \\
    & $L_{\infty}$ & 1.23e-5 & 1.44e-5 & 6.15e-5 & 6.00e-4  \\
\hline
4   & $L_2$        & 1.36e-5 & 1.26e-5 & 8.48e-6 & 2.08e-5  \\
    & $L_{\infty}$ & 3.67e-5 & 3.67e-5 & 1.69e-4 & 1.53e-3  \\
\hline
\end{tabular}
\end{table}

In Fig. \ref{fig:sol}, we present the scaled TW-PINN solution at the final time specified in Table \ref{tab:domain} for each equation with $\rho=10^6$, compared to the corresponding exact solution.
For each equation, the solution is obtained from a representative scaled TW-PINN on the original domain.
It is seen that our scaled TW-PINN accurately captures the sharp wave front.
\begin{figure}[ht!]
\centering
\makebox[\columnwidth][c]{
\makebox[0.25\columnwidth][c]{\hspace{10mm} \scriptsize Fisher's}
\makebox[0.25\columnwidth][c]{\hspace{7mm} \scriptsize NWS $(q=2)$}
\makebox[0.25\columnwidth][c]{\hspace{3mm} \scriptsize Zeldovich}
\makebox[0.25\columnwidth][c]{\hspace{-1mm} \scriptsize bistable $(a=0.2)$}
}
\includegraphics[width=1\columnwidth]{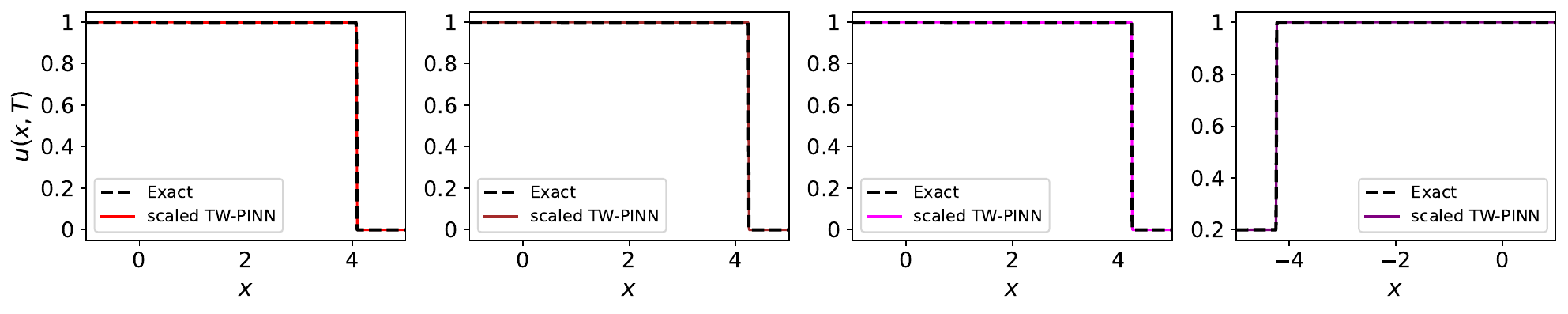}
\caption{Solution profiles, from left to right, for Fisher's equation at $T = 0.002$, NWS equation $(q=2)$ at $T = 0.002$, Zeldovich equation at $T = 0.006$, and bistable equation $(a=0.2)$ at $T = 0.005$. The dashed black line is the exact solution.}
\label{fig:sol}
\end{figure}

To further illustrate the error behavior, we plot the maximum error over the spatial domain against time $t$ in Fig. \ref{fig:max_error_all}, which shows the temporal evolution of this maximum error on a logarithmic scale for the Fisher's, NWS, Zeldovich and bistable equations.
Unlike classical numerical theory, where the error typically accumulates over time, the maximum error in scaled TW-PINN solutions does not necessarily increase monotonically.
\begin{figure}[ht!]
\centering
\makebox[\columnwidth][c]{
\makebox[0.25\columnwidth][c]{\hspace{10mm} \scriptsize Fisher's}
\makebox[0.25\columnwidth][c]{\hspace{7mm} \scriptsize NWS $(q=2)$}
\makebox[0.25\columnwidth][c]{\hspace{3mm} \scriptsize Zeldovich}
\makebox[0.25\columnwidth][c]{\hspace{-1mm} \scriptsize bistable $(a=0.2)$}
}
\includegraphics[width=\columnwidth]{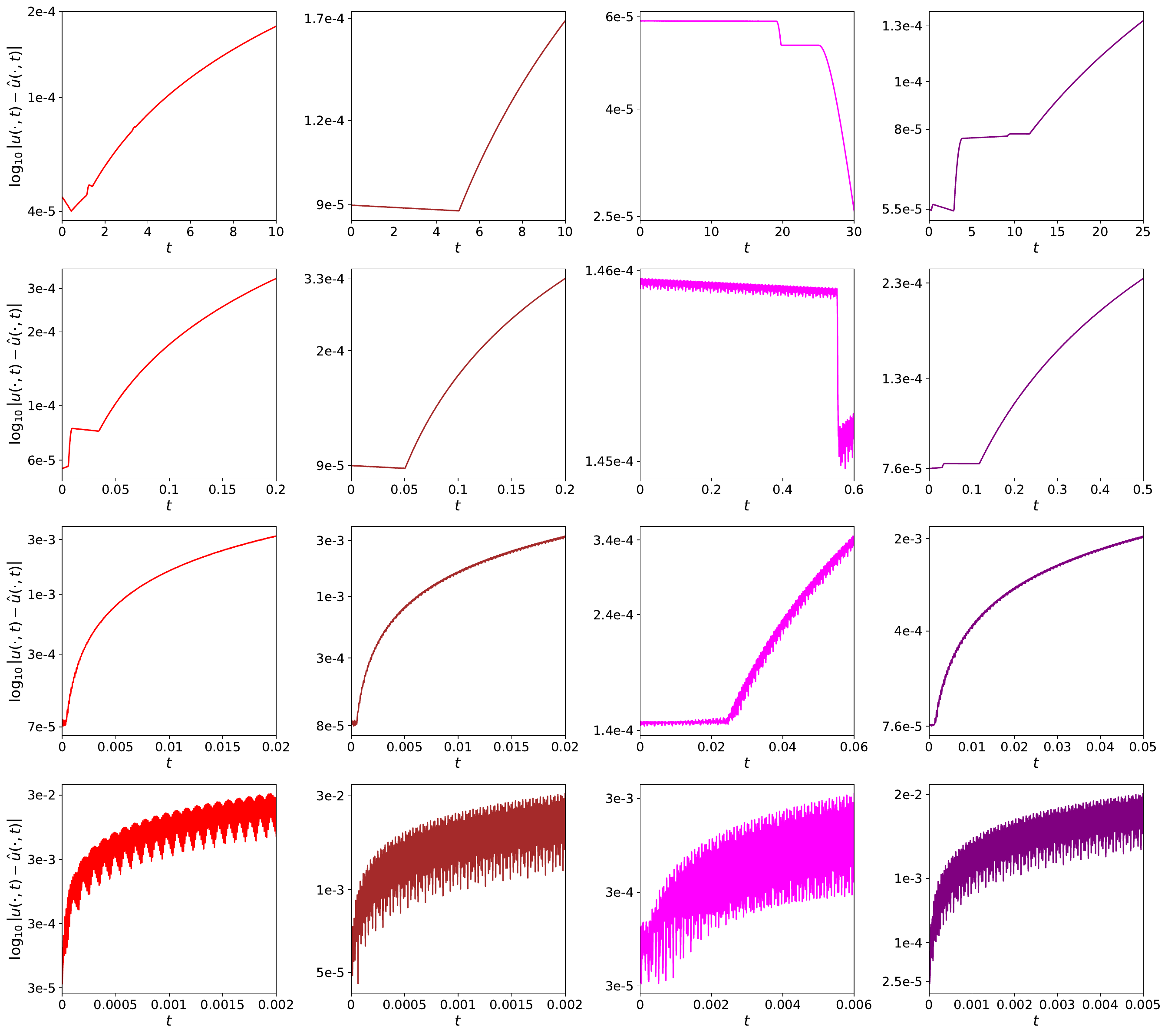}
\caption{Evolution of log-scale maximum errors over time, from left to right, for Fisher's, NWS $(q=2)$, Zeldovich, and bistable $(a=0.2)$ equations. The rows correspond to $\rho=1,10^2,10^4,$ and $10^6$ from top to bottom.}
\label{fig:max_error_all}
\end{figure}

\subsection{Two-dimensional numerical results}
\label{sec:result2D}
In two-dimensional problems, we evaluate the $L_2$ and $L_\infty$ errors by varying the reaction coefficient $\rho$, given two propagation direction vectors $\bfn$.
The spatial and temporal domains for each reaction-diffusion equation and each value of $\rho$ are chosen to be the same as in one-dimensional problems, listed in Table \ref{tab:domain}.

For each value of $\rho$, $100$ collocation points are uniformly sampled in both the spatial and temporal directions, generating a total of $10^6$ collocation points.
We use the same PINN solvers trained on the original and restricted domains as in Section \ref{sec:result1D}.
The $L_2$ and $L_{\infty}$ errors, in terms of mean and standard deviation, for Fisher’s, NWS, Zeldovich and bistable equations at $\rho$ and $\bfn$ are computed in Tables \ref{tab:L2_2D} and \ref{tab:Linf_2D}, respectively. 
The propagation direction $\bfn$ is selected from $\{ (1,1), (1,3) \}$ and normalized to unit length.
Since the direction $(3,1)$ is symmetric to $(1,3)$, it produces equivalent results and is omitted.
The error grows generally with $\rho$, consistent with the one-dimensional results, while showing little variation with respect to $\bfn$.
This is expected because the same PINN solver for stage two is used in both one- and two-dimensional problems, and the remaining difference comes only from $\bfn$ and the number of collocation points.
\begin{table}[ht!]
\centering
\caption{$L_2$ error for two-dimensional Fisher's, NWS $(q=2)$, Zeldovich and bistable $(a=0.2)$ equations.}
\label{tab:L2_2D}
\resizebox{\textwidth}{!}{
\begin{tabular}{l l c c c c c}
\hline
          &            & $\bfn$ & \multicolumn{4}{c}{$\rho$}   \\
                                  \cline{4-7}
          &            &        & $1$               & $10^2$            & $10^4$            & $10^6$ \\
\hline
Fisher    & original   & (1,1)  & 5.07e-5 (2.63e-5) & 6.80e-5 (3.82e-5) & 1.91e-4 (1.28e-4) & 5.35e-4 (3.61e-4) \\
          &            & (1,3)  & 5.02e-5 (2.46e-5) & 6.58e-5 (3.57e-5) & 1.81e-4 (1.21e-4) & 5.67e-4 (3.82e-4) \\
          \cline{2-7}
          & restricted & (1,1)  & 5.12e-6 (1.50e-6) & 6.50e-6 (1.36e-6) & 6.36e-6 (1.48e-6) & 9.95e-6 (5.18e-6) \\
          &            & (1,3)  & 5.14e-6 (1.50e-6) & 6.94e-6 (1.53e-6) & 6.36e-6 (1.47e-6) & 1.04e-5 (5.55e-6) \\
\hline
NWS       & original   & (1,1)  & 1.20e-4 (8.98e-5) & 1.65e-4 (1.29e-4) & 5.60e-4 (4.54e-4) & 5.28e-3 (4.17e-3) \\
          &            & (1,3)  & 1.16e-4 (8.60e-5) & 1.58e-4 (1.23e-4) & 4.87e-4 (3.95e-4) & 1.53e-3 (1.25e-3) \\
          \cline{2-7}
          & restricted & (1,1)  & 6.29e-6 (2.37e-6) & 6.36e-6 (2.19e-6) & 6.19e-6 (2.31e-6) & 3.70e-5 (2.15e-5) \\
          &            & (1,3)  & 6.21e-6 (2.37e-6) & 6.27e-6 (2.19e-6) & 5.81e-6 (2.08e-6) & 1.14e-5 (5.97e-6) \\
\hline
Zeldovich & original   & (1,1)  & 7.04e-5 (2.48e-5) & 8.08e-5 (3.83e-5) & 2.33e-4 (1.59e-4) & 2.24e-3 (1.54e-3) \\
          &            & (1,3)  & 6.81e-5 (2.38e-5) & 7.80e-5 (3.64e-5) & 2.03e-4 (1.37e-4) & 6.34e-4 (4.34e-4) \\
          \cline{2-7}
          & restricted & (1,1)  & 1.11e-5 (5.29e-6) & 1.07e-5 (4.87e-6) & 1.08e-5 (9.44e-6) & 6.95e-5 (9.55e-5) \\
          &            & (1,3)  & 1.08e-5 (4.92e-6) & 1.04e-5 (4.65e-6) & 1.01e-5 (8.08e-6) & 2.14e-5 (2.64e-5) \\
\hline
bistable  & original   & (1,1)  & 5.25e-5 (3.28e-5) & 5.80e-5 (4.46e-5) & 1.38e-4 (1.58e-4) & 1.25e-3 (1.45e-3) \\
          &            & (1,3)  & 5.14e-5 (3.16e-5) & 5.63e-5 (4.26e-5) & 1.27e-4 (1.46e-4) & 3.96e-4 (4.62e-4) \\
          \cline{2-7}
          & restricted & (1,1)  & 7.39e-6 (3.15e-6) & 7.90e-6 (2.82e-6) & 7.05e-6 (2.19e-6) & 3.13e-5 (2.07e-5) \\
          &            & (1,3)  & 7.19e-6 (2.97e-6) & 7.66e-6 (2.59e-6) & 6.80e-6 (2.12e-6) & 1.18e-5 (5.65e-6) \\
\hline
\end{tabular}}
\end{table}
\begin{table}[ht!]
\centering
\caption{$L_{\infty}$ error for two-dimensional Fisher's, NWS $(q=2)$, Zeldovich and bistable $(a=0.2)$ equations.}
\label{tab:Linf_2D}
\resizebox{\textwidth}{!}{
\begin{tabular}{l l c c c c c}
\hline
          &            & $\bfn$ & \multicolumn{4}{c}{$\rho$}   \\
                                  \cline{4-7}
          &            &        & $1$               & $10^2$            & $10^4$            & $10^6$ \\
\hline
Fisher    & original   & (1,1)  & 1.80e-4 (9.90e-5) & 3.26e-4 (2.07e-4) & 3.10e-3 (2.09e-3) & 2.45e-2 (1.66e-2) \\
          &            & (1,3)  & 1.80e-4 (9.90e-5) & 3.27e-4 (2.07e-4) & 3.14e-3 (2.12e-3) & 2.96e-2 (2.00e-2) \\
          \cline{2-7}
          & restricted & (1,1)  & 1.84e-5 (5.43e-6) & 1.91e-5 (5.13e-6) & 4.97e-5 (3.57e-5) & 3.55e-4 (2.86e-4) \\
          &            & (1,3)  & 1.84e-5 (5.49e-6) & 1.91e-5 (5.13e-6) & 4.98e-5 (3.57e-5) & 4.31e-4 (3.49e-4) \\
\hline
NWS       & original   & (1,1)  & 5.21e-4 (4.05e-4) & 1.04e-3 (8.29e-4) & 1.04e-2 (8.42e-3) & 9.92e-2 (7.69e-2) \\
          &            & (1,3)  & 5.22e-4 (4.05e-4) & 1.04e-3 (8.29e-4) & 1.02e-2 (8.25e-3) & 9.40e-2 (7.73e-2) \\
          \cline{2-7}
          & restricted & (1,1)  & 1.36e-5 (5.14e-6) & 1.40e-5 (5.00e-6) & 7.05e-5 (4.10e-5) & 7.01e-4 (4.11e-4) \\
          &            & (1,3)  & 1.36e-5 (5.14e-6) & 1.40e-5 (5.01e-6) & 6.99e-5 (4.01e-5) & 6.28e-4 (3.68e-4) \\
\hline
Zeldovich & original   & (1,1)  & 2.42e-4 (1.27e-4) & 4.38e-4 (2.82e-4) & 4.30e-3 (2.95e-3) & 4.28e-2 (2.92e-2) \\
          &            & (1,3)  & 2.42e-4 (1.27e-4) & 4.39e-4 (2.82e-4) & 4.22e-3 (2.90e-3) & 3.85e-2 (2.64e-2) \\
          \cline{2-7}
          & restricted & (1,1)  & 2.65e-5 (1.32e-5) & 2.77e-5 (1.65e-5) & 1.32e-4 (1.83e-4) & 1.31e-3 (1.83e-3) \\
          &            & (1,3)  & 2.65e-5 (1.32e-5) & 2.79e-5 (1.67e-5) & 1.30e-4 (1.79e-4) & 1.17e-3 (1.64e-3) \\
\hline
bistable  & original   & (1,1)  & 1.58e-4 (1.27e-4) & 2.63e-4 (2.67e-4) & 2.39e-3 (2.79e-3) & 2.37e-2 (2.75e-2) \\
          &            & (1,3)  & 1.58e-4 (1.28e-4) & 2.64e-4 (2.67e-4) & 2.36e-3 (2.76e-3) & 2.15e-2 (2.50e-2) \\
          \cline{2-7}
          & restricted & (1,1)  & 1.73e-5 (7.68e-6) & 1.75e-5 (7.52e-6) & 6.00e-5 (3.89e-5) & 5.73e-4 (4.18e-4) \\
          &            & (1,3)  & 1.73e-5 (7.68e-6) & 1.75e-5 (7.52e-6) & 5.97e-5 (3.87e-5) & 5.13e-4 (3.75e-4) \\
\hline
\end{tabular}}
\end{table}

Figure \ref{fig:err2D} presents the absolute error between the exact and numerical solutions at the final time for each equation with $\rho=10^4, 10^6$ and $\bfn=(1,1)$.
We use a representative scaled TW-PINN on the original domain.
\begin{figure}[ht!]
\centering
\makebox[\columnwidth][c]{
\makebox[0.25\columnwidth][c]{\hspace{3mm} \scriptsize Fisher's}
\makebox[0.25\columnwidth][c]{\hspace{1mm} \scriptsize NWS $(q=2)$}
\makebox[0.25\columnwidth][c]{\hspace{-2mm} \scriptsize Zeldovich}
\makebox[0.25\columnwidth][c]{\hspace{-5mm} \scriptsize bistable $(a=0.2)$}
}
\includegraphics[width=\columnwidth]{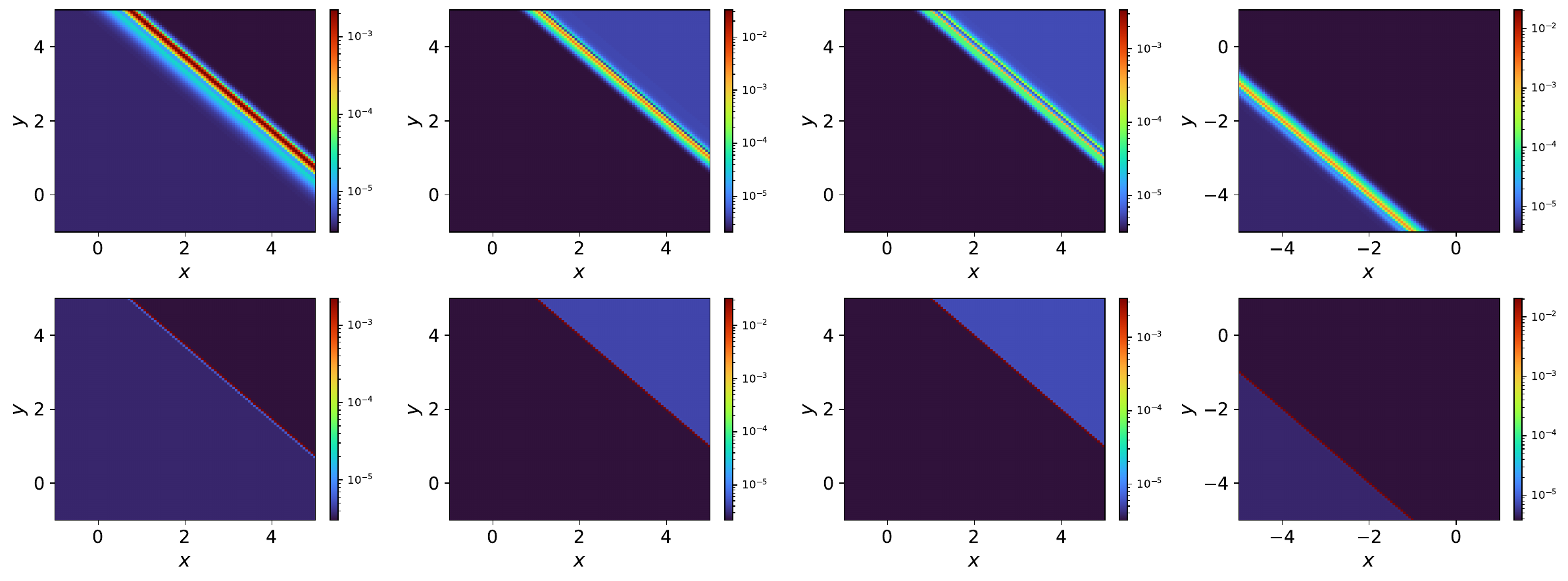}
\caption{Absolute error between the exact and scaled TW-PINN solutions in the filled contour plot, from left to right, for Fisher's equation at $T = 0.002$, NWS equation $(q=2)$ at $T = 0.002$, Zeldovich equation at $T = 0.006$, and bistable equation $(a=0.2)$ at $T = 0.005$ with $\rho=10^4$ (top) and $\rho=10^6$ (bottom).}
\label{fig:err2D}
\end{figure}
As shown in Figure \ref{fig:err2D}, larger errors occur near the wave front, but the error remains small, which means that the scaled TW-PINN captures the traveling wave front accurately.
We also compute the maximum spatial error over $t$, but omit the resulting curves because they are similar to those in Fig. \ref{fig:max_error_all} for the one-dimensional case.

\subsection{Comparison with wave-PINN}
\label{sec:comparison}
We compare the proposed scaled TW-PINN with the existing wave-PINN method in \cite{wavePINN}, which was designed specifically for the Fisher’s equation with the traveling wave solution.
The wave layer for wave-PINN defines the traveling wave coordinate as
\begin{equation*}
\hat{z} = \omega_1 \! \sqrt{\rho}\, x + \omega_2 \, \rho\, t + \omega_3.
\end{equation*}
Then the reaction coefficient is incorporated into the network input, enabling approximation for different $\rho$ while $D$ is fixed.
Consequently the wave-PINN must be retrained if $D$ changes, which increases computational cost and limits the reusability of the trained network.
Moreover, this coordinate definition is applicable only to one spatial dimension because it directly uses the input variables $x$ and $t$ without a mechanism for extension to higher dimensions.
In contrast, our wave layer \eqref{eq:wave_layer} operates on scaled variables and is independent of spatial dimension. 
Then a single scaled TW-PINN can handle various physical coefficients and multiple spatial dimensions.
Figure \ref{fig:wave_layer} compares the the structure of the wave layer in wave-PINN and scaled TW-PINN.
\begin{figure}[ht!]
\centering
\includegraphics[width=0.9\columnwidth]{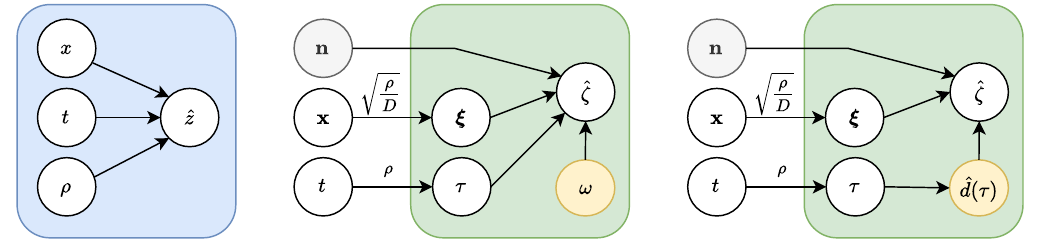}
\caption{Comparison of the structure of the wave layer in wave-PINN (left), scaled TW-PINN (middle) and scaled gTW-PINN (right).}
\label{fig:wave_layer}
\end{figure}
In addition, wave-PINN and scaled TW-PINN differ in the number of hidden layers: the former uses three hidden layers, while the latter uses only one.
For a fair comparison, wave-PINN is trained under a configuration similar to that of the proposed scaled TW-PINN while keeping its original architecture.
The network contains three hidden layers with $20$ neurons each and applies the logistic sigmoid activation function.
In training, we take $1024$ samples from $\rho \in (1,10^6)$ and the initial learning rate is $0.001$.

Table \ref{tab:pinn_compare1D} presents the $L_2$ and $L_{\infty}$ errors of wave-PINN and scaled TW-PINN on the original domain for the Fisher's equation.
\begin{table}[ht!]
\centering
\caption{$L_2$ and $L_{\infty}$ errors of wave-PINN and scaled TW-PINN for Fisher's equation.}
\label{tab:pinn_compare1D}
\resizebox{\textwidth}{!}{
\begin{tabular}{l l c c c c }
\hline
              & norm         & \multicolumn{4}{c}{$\rho$} \\
                               \cline{3-6}
              &              & $1$               & $10^2$            & $10^4$            & $10^6$  \\
\hline
wave-PINN     & $L_2$        & 2.23e-3 (1.67e-3) & 3.11e-3 (2.35e-3) & 9.72e-3 (7.32e-3) & 2.37e-2 (1.44e-2) \\ 
              & $L_{\infty}$ & 8.95e-3 (6.75e-3) & 1.77e-2 (1.34e-2) & 1.71e-1 (1.26e-1) & 7.64e-1 (2.94e-1) \\ 
scaled TW-PINN & $L_2$        & 5.21e-5 (2.41e-5) & 6.60e-5 (3.46e-5) & 1.76e-4 (1.18e-4) & 5.55e-4 (3.75e-4)  \\
              & $L_{\infty}$ & 1.80e-4 (9.90e-5) & 3.27e-4 (2.07e-4) & 3.14e-3 (2.12e-3) & 3.06e-2 (2.07e-2) \\ 
\hline
\end{tabular}}
\end{table}
As shown in Table \ref{tab:pinn_compare1D}, our scaled TW-PINN achieves significantly higher accuracy for all values of $\rho$.
The approximate solutions for $\rho=10^6$ at the initial time $t=0$ and the final time $T=0.002$, along with the pointwise errors on a logarithmic scale, are plotted in Fig. \ref{fig:visual}.
\begin{figure}[ht!]
\centering
\includegraphics[width=0.9\columnwidth]{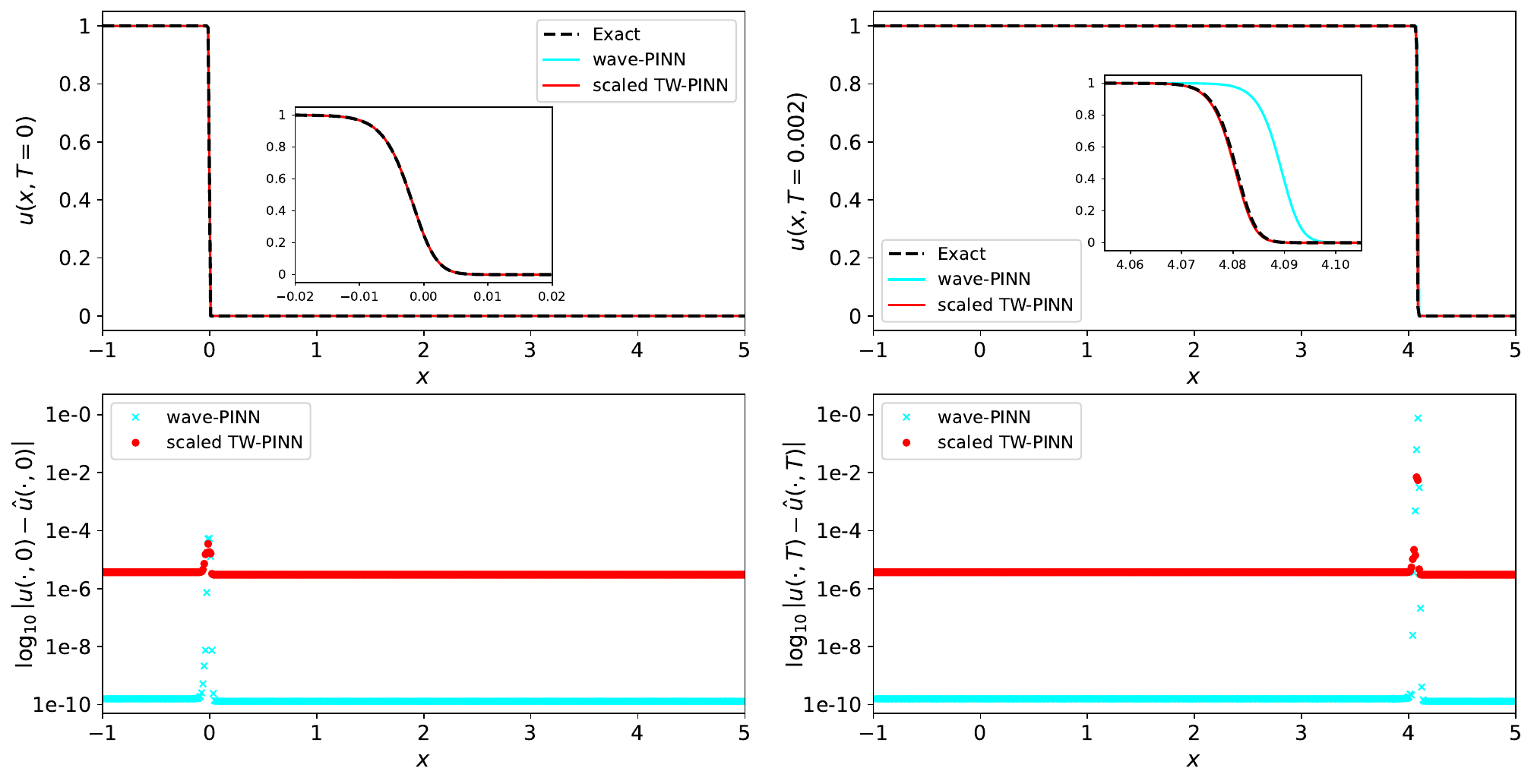}
\caption{Solution profiles (top) and log-scale pointwise errors (bottom) for Fisher's equation with $\rho=10^6$ at the initial time $t=0$ (left) and the final time $T = 0.002$ (right) approximated by wave-PINN (cyan) and scaled TW-PINN (red). The dashed black line is the exact solution.}
\label{fig:visual}
\end{figure}
We can see that both wave-PINN and scaled TW-PINN produce accurate solutions at the initial time $t=0$.
However, at the final time, wave-PINN predicts the wave front with an excessive speed, causing the front to move ahead of its correct location, whereas the proposed scaled TW-PINN provides a more accurate approximation.
While wave-PINN achieves higher accuracy near the equilibrium states, it exhibits substantially larger errors in the transition layer.
In contrast, scaled TW-PINN shows slightly larger errors near the equilibrium states but maintains stable accuracy across the transition layer.
These differences appear to arise from the residual weighting scheme in wave-PINN. 
By assigning smaller weights to the sharp region, wave-PINN tends to learn the equilibrium regions more accurately than the wave front. 
In contrast, the scaling transformation in scaled TW-PINN smooths the sharp transition, allowing the PINN solver in stage two to achieve more balanced learning across the spatial domain.
These results demonstrate that scaled TW-PINN more reliably predicts the traveling wave behavior than wave-PINN, particularly in terms of wave front position and wave speed.

\section{Extension to Fisher's equation with general initial conditions} \label{sec:general}
Previously, we only consider the specific initial condition for which the reaction-diffusion equation admits an exact traveling wave solution in closed form with a special wave speed.
We now turn to more general initial conditions.

We focus on the one-dimensional scaled Fisher's equation,
\begin{equation*}
v_{\tau} = v_{\xi \xi} + v(1-v),
\end{equation*}
subject to the initial condition $v(\xi,0) = v_0(\xi)$.
Notice that the Fisher's equation is chosen as its long time behavior of the traveling wave solution under the general initial condition is extensively studied in the literature \cite{Bramson,MurrayI}.
The same framework also extends to the other equations in this paper.
Suppose that the initial condition satisfies
\begin{equation}
\label{eq:gen_IC}
v_0(\xi) \sim \exp (-\lambda\xi) \quad \text{as } \: \xi \to \infty,
\end{equation}
with $\lambda>0$.
By \cite{MurrayI}, the solution evolves to a traveling wave whose asymptotic wave speed is
\begin{equation*}
c(\lambda) =
\left\{
\begin{array}{ll}
\lambda+\frac{1}{\lambda}, & 0 < \lambda < 1, \\
                        2, & \lambda \geq 1.
\end{array}
\right.
\end{equation*}
For $0 < \lambda < 1$, the asymptotic behavior is not fully characterized.
In the case $\lambda \geq 1$ corresponding to the minimum wave speed $c=2$, the long time behavior of the solution \cite{Bramson} is described by
\begin{equation}
\label{eq:asymptotic}
v(\xi,\tau) \to V \! \left( \xi - 2 \tau + \tfrac{3}{2} \ln \tau + O(1) \right) \quad \text{as } \tau \to \infty,
\end{equation}
with $O(1)$ a bounded spatial shift. 
Thus, the wave speed is not constant.
Accordingly, a generalized traveling wave coordinate is introduced as
\begin{equation*}
\zeta = \xi - d(\tau),
\end{equation*}
where $d(\tau)$ represents a time-dependent shift and the wave speed is given by $d'(\tau)$.
For the constant wave speed, one sets $d(\tau) = c \tau$, which gives $d'(\tau) = c$.
Based on this formulation, the wave layer \eqref{eq:wave_coordinate_predict} is generalized to 
\begin{equation*}
\hat{\zeta} = \xi - \hat{d}(\tau),
\end{equation*}
as shown on the right of Fig. \ref{fig:wave_layer}.
Following \eqref{eq:asymptotic} gives the definition of the predicted wave shift
\begin{equation}
\label{eq:wave_coordinate_predict_general}
\hat{d}(\tau;\lambda) = c(\lambda) \tau - \tfrac{3}{2} \ln (\tau+1) + w \lambda,
\end{equation}
where $\lambda$ corresponds to the exponential decay of the initial condition \eqref{eq:gen_IC}, and $c(\lambda)$ denotes the asymptotic wave speed \eqref{eq:asymptotic}.
The term $\ln (\tau+1)$ avoids the singularity at $\tau=0$ while preserving the same asymptotic behavior for large $\tau$, and the final term $w\lambda$ introduces a trainable correction that accounts for the bounded spatial shift $O(1)$.
Although the asymptotic behavior is known only for $\lambda \geq 1$, we adopt the form \eqref{eq:asymptotic} here as a heuristic approximation and investigate whether it remains informative for describing the wave shift under $0<\lambda<1$.

We consider three initial conditions:
(i) a step function
\begin{equation*}
v_0 (\xi)=
\left\{
\begin{array}{ll}
1, & \xi < 0, \\
0, & \xi \geq 0,
\end{array}
\right.
\end{equation*}
(ii) a logistic-type function of the form
\begin{equation}
\label{eq:fisher_IC}
v_0(\xi) = \frac{1}{\left[ 1 + \exp(\frac{\lambda}{2} \xi) \right]^2},
\end{equation}
with $\lambda=2$, and
(iii) the logistic-type function \eqref{eq:fisher_IC} with $\lambda=\tfrac12$.
The two values of $\lambda$ correspond to $\lambda \geq 1$ and $0 < \lambda < 1$, respectively.
The initial condition of the step function has a rapidly decaying leading edge and behaves similarly to initial conditions with $\lambda \geq 1$.

We take the PINN architecture as in Section \ref{sec:architecture} with the wave layer replaced by \eqref{eq:wave_coordinate_predict_general}.
The same loss function in Section \ref{sec:loss} is employed, and Dirichlet boundary conditions are imposed according to \eqref{eq:equil_state}.
Since a relatively small training domain with the focus on the wave front promotes rapid physical convergence, we choose the training $(\xi, \tau)$-domain as $[-300,900] \times [0,300]$, and pick $N_{\ICBC} = 1024$ and $N_{\resid} = 1024$ collocation points for $\calL_{\ICBC}$ and $\mathcal{L}_{\resid}$, respectively.
The generalized PINN solver is trained for $30,000$ epochs.
At early times, the solution deforms from the initial condition.
Specifically, points along the wave front propagate at different speeds depending on the initial condition, resulting in a deformation of the wave front shape.
For initial conditions (i) and (ii), matching $\lambda \geq 1$, the wave front propagates more rapidly near its leading edge, whereas in the initial condition (iii), corresponding to $0 < \lambda < 1$, the front moves faster near its trailing edge.
Since the PINN solver is designed for traveling waves, it is not easy to capture such deformation.
To address this, the initial and boundary loss, as well as the residual loss, is applied during the first $0.3$ epochs of training, after which only the residual loss is used.
All other training settings are the same as in Section \ref{sec:training}.

The scaling PINN framework for general initial conditions, termed scaled gTW-PINN, is used in the following numerical experiments.
We consider the Fisher's equation with $\rho=10^2$ and $\rho=10^4$, and approximate the corresponding scaled gTW-PINN solutions at final times $T=0.3$ and $T=0.03$.
Since the exact solution is not available for Fisher's equation with these general initial conditions, the reference solution is computed by the central WENO scheme \cite{Rathan,Gu} with $2000$ uniform spatial cells on the domain $[-3, 9]$.
Figure \eqref{fig:ini_sol2} shows the numerical results for $\rho=10^2$. 
We see that scaled gTW-PINN captures the wave front well for the initial conditions (i) and (ii) with $\lambda \geq 1$, validating the asymptotic behavior \eqref{eq:asymptotic}.
For the initial condition (iii), the solution with scaled gTW-PINN travels behind the reference wave front, and the phase error is noticeable.
This discrepancy may be caused by the fact that the predicted wave shift \eqref{eq:wave_coordinate_predict_general} from the asymptotic form \eqref{eq:asymptotic}, is not well suited to the case $0 < \lambda < 1$. 
In particular, the logarithmic phase correction and the accompanying $O(1)$ spatial shift may introduce an inaccurate wave shift.
We plot the numerical solution for $\rho=10^4$ by scaled gTW-PINN together with the reference solution in Fig. \eqref{fig:ini_sol2}.
It is observed that the wave front approximated by scaled gTW-PINN agrees with the reference solution for all three initial conditions.
\begin{figure}[ht!]
\centering
\includegraphics[width=0.9\columnwidth]{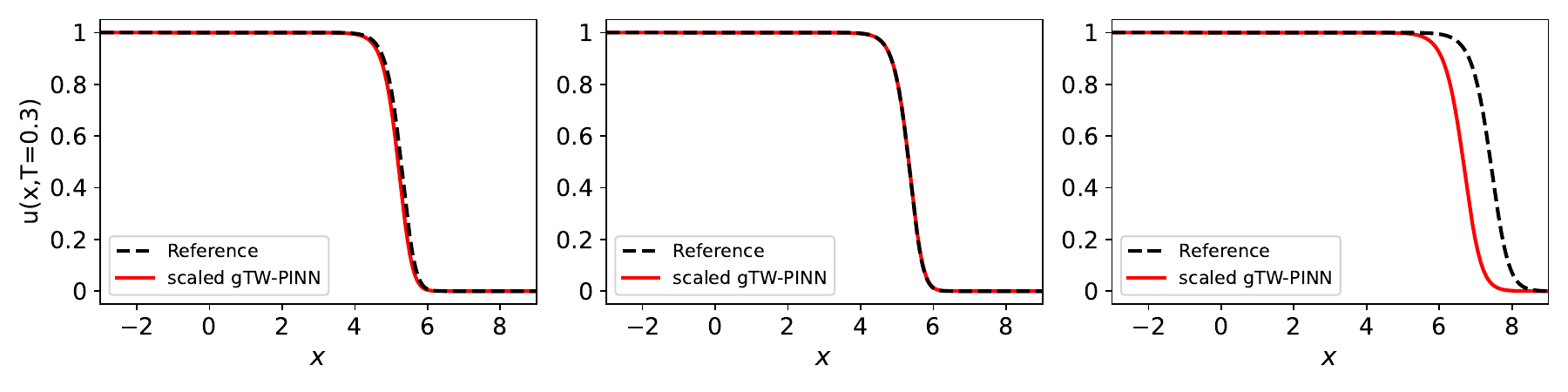}
\caption{Solution profiles for Fisher's equation ($\rho=10^2$) subject to the initial condition (i) the step function (left), (ii) \eqref{eq:fisher_IC} with $\lambda=2$ (middle) and (iii) \eqref{eq:fisher_IC} with $\lambda = \tfrac12$ (right) at the final time $T=0.3$ approximated by scaled gTW-PINN (red). 
The dashed black line is the reference solution.}
\label{fig:ini_sol2}
\end{figure}
\begin{figure}[ht!]
\centering
\includegraphics[width=0.9\columnwidth]{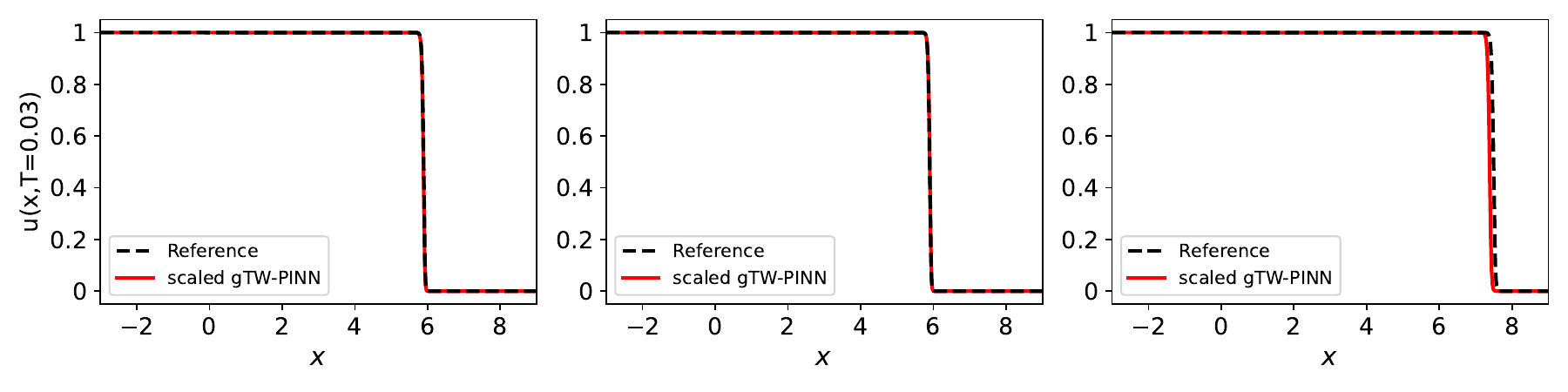}
\caption{Solution profiles for Fisher's equation ($\rho=10^4$) subject to the initial condition (i) the step function (left), (ii) \eqref{eq:fisher_IC} with $\lambda=2$ (middle) and (iii) \eqref{eq:fisher_IC} with $\lambda = \tfrac12$ (right) at the final time $T=0.03$ approximated by scaled gTW-PINN (red). 
The dashed black line is the reference solution.}
\label{fig:ini_sol4}
\end{figure}

\section{Conclusion} \label{sec:conclusion}
In this paper, we develop a scaling PINN framework to solve the reaction–diffusion equation with the traveling wave solution. 
For the sharp wave front from the large reaction coefficient, we apply a scaling transformation that normalizes the reaction and diffusion coefficients, employ a PINN solver for the resulting scaled equation, and recover the solution of the original equation via an inverse transformation.
The wave layer in the PINN solver is designed to preserve the traveling wave form, even in higher spatial dimensions.
We also show that the proposed PINN solver has a universal approximation property for traveling wave solutions.
One- and two-dimensional numerical experiments demonstrate that the proposed scaled TW-PINN accurately captures the traveling wave solution, and the comparison to the existing wave-PINN method show that it achieves higher accuracy and smaller errors near the wave front.
Furthermore, the framework extends to general initial conditions, suggesting its broad applicability.
Future work will investigate the application of PINNs with scaling and inverse transformations to reaction–diffusion systems exhibiting more complex phenomena, such as spiral and scroll waves.

\section*{Acknowledgments}
Jiaxi Gu is supported by POSTECH Basic Science Research Institute Fund, whose NRF grant number is RS-2021-NR060139.
Jae-Hun Jung is supported by National Research Foundation of Korea (NRF) under the grant number 2021R1A2C3009648, POSTECH Basic Science Research Institute under the NRF grant number 2021R1A6A1A10042944, and partially NRF grant funded by the Korea government (MSIT) (RS-2023-00219980).


\providecommand{\bysame}{\leavevmode\hbox to3em{\hrulefill}\thinspace}
\providecommand{\MR}{\relax\ifhmode\unskip\space\fi MR }
\providecommand{\MRhref}[2]{%
  \href{http://www.ams.org/mathscinet-getitem?mr=#1}{#2}
}
\providecommand{\href}[2]{#2}

\end{document}